\documentclass[11pt]{amsart}
\usepackage{amssymb}
\usepackage[T1]{fontenc}
\usepackage{hyperref}

\usepackage{etoolbox}
\patchcmd{\section}{\scshape}{\bfseries}{}{}
\newcommand{\sectionbf}[1]{%
  \stepcounter{section}%
  \section*{\thesection.~#1}%
}
    
\numberwithin{equation}{section}

\theoremstyle{plain}
\newtheorem{theorem}{Theorem}
\newtheorem{lemma}{Lemma}[section]
\newtheorem{proposition}[lemma]{Proposition}
\newtheorem*{openproblem}{Open problem}
\newtheorem*{theorem-nonumber}{Theorem}
\theoremstyle{definition}
\newtheorem*{definition}{Definition}

\newcommand{\area}{A}
\newcommand{\cl}[1]{\overline{#1}}
\newcommand{\defeq}{\mathrel{\mathop:}=}
\newcommand{\dist}{d}
\newcommand{\geodflow}{\Phi}
\newcommand{\mathreal}{\mathbb{R}}
\newcommand{\mrealnneg}{\mathreal_{\ge 0}}
\newcommand{\mrealpos}{\mathreal_{>0}}

\newcommand{\mathinteger}{\mathbb{Z}}
\newcommand{\mhdf}{\mathcal{H}}
\newcommand{\mllb}{\nu}
\newcommand{\mlle}{\mu}
\newcommand{\nv}{\langle v,N\circ\sigma(v) \rangle}
\renewcommand{\setminus}{\smallsetminus}

\begin{document}
\title[Area growth and rigidity of surfaces]{Area growth and rigidity of surfaces without conjugate points}
\author[V.~Bangert \& P.~Emmerich]{Victor Bangert \& Patrick Emmerich}
\begin{abstract}
  We prove flatness of complete Riemannian planes and cylinders without conjugate points under optimal conditions on the area growth.
\end{abstract}
\maketitle

\sectionbf{Introduction}

\noindent{}In 1942 M. Morse and G.~A. Hedlund \cite{MH} conjectured that every Riemannian $2$-torus without conjugate points is flat. This was proved by E.~Hopf in 1943, see \cite{HOPF}. The natural question, if Riemannian tori without conjugate points and of arbitrary dimension are flat, was answered affirmatively by D. Burago and S. Ivanov \cite{BI}, by a completely new method.

Here, we apply E. Hopf's original method to the study of complete Riemannian planes and cylinders without conjugate points. In these cases one needs additional assumptions to prove flatness since the plane and the cylinder admit complete Riemannian metrics with non-positive curvature (and, hence, without conjugate points) that are not flat. In this situation, conditions on the area growth are particularly natural. For the case of the plane we prove the following optimal result.

\begin{theorem}\label{thm:plane}
  Let $g$ be a complete Riemannian metric without conjugate points on the plane $\mathreal^2$. Then, for every $p\in\mathreal^2$, the area $A_p(r)$ of the metric ball with center $p$ and radius $r$ satisfies
  \[
    \liminf_{r\to\infty}\frac{A_p(r)}{\pi r^2}\geq 1    
  \]
with equality if and only if $g$ is flat.
\end{theorem}

Note that, for every $\epsilon>0$, one can easily find complete planes with non-positive curvature and conical end such that $\lim_{r\to\infty}A_p(r) / \pi r^2=1+\epsilon$. These examples show that the estimate in Theorem~\ref{thm:plane} is optimal. Note moreover, that Theorem~\ref{thm:plane} does not follow from the following well-known conjecture on the area of small disks, see \cite{CROKE09}: If a metric ball $B(p,r)$ with center $p$ and radius $r$ in a complete Riemannian surface has no conjugate points, then $A_p(r)\geq \frac{8}{\pi} r^2$ with equality if and only if $B(p,r)$ is isometric to a hemisphere of radius $r$.

To state our rigidity result for cylinders we first define what it means that an end of a cylinder has subquadratic area growth. As usual, let $A$ denote the area induced by the Riemannian metric.

\begin{definition}
  Let $S$ be a complete, connected Riemannian surface. An end $\mathcal{E}$ of $S$ has \emph{subquadratic area growth} if there exists a neighborhood $U\subseteq S$ of $\mathcal{E}$ such that
  \[
    \liminf_{r\to\infty}\frac{A(U\cap B(p,r))}{r^2}=0
  \]
for one (and hence every) point $p\in S$.
\end{definition}

\begin{theorem}\label{thm:A}
  Let $g$ be a complete Riemannian metric without conjugate points on the cylinder $S^1\times\mathreal$. If both ends of the cylinder have subquadratic area growth then $g$ is flat.
\end{theorem}

There is an alternative version of Theorem~\ref{thm:A} that involves an assumption on the growth of the lengths of shortest non-contractible loops. Actually, we will prove in Section 8 that the two versions are equivalent. We let $d$ denote the distance induced by the Riemannian metric.

\begin{definition}
  Let $C=S^1\times\mathreal$ be a complete Riemannian cylinder and, for $p\in C$, let $l(p)$ denote the length of a shortest non-contractible loop based at $p$. We say that an end $\mathcal{E}$ of $C$ \emph{opens less than linearly} if there exists a sequence $(p_i)$ in $C$ converging to $\mathcal{E}$ such that
  \[
    \lim_{i\to\infty}\frac{l(p_i)}{\dist (p_i,p_0)}=0.
  \]
\end{definition}

\begingroup
\addtocounter{theorem}{-1}
\def\thetheorem{\arabic{theorem}'}
\begin{theorem}\label{thm:l}
  Let $g$ be a complete Riemannian metric without conjugate points on the cylinder $S^1\times\mathreal$. If both ends of the cylinder open less than linearly then $g$ is flat.
\end{theorem}
\endgroup

Again, simple examples of cylinders of revolution with non-positive Gausian curvature and conical ends show that the conditions in Theorem~\ref{thm:A} and \ref{thm:l} are optimal, see \cite[Section~1]{BE11}.

Rigidity results of the type of Theorem~\ref{thm:l} have been proved by K. Burns and G. Knieper \cite{BK}, H.~Koehler \cite{KOEHLER}, and by the present authors \cite{BE11}. All of these involve stronger conditions on the growth of $l$ and additional conditions on the Gaussian curvature. So they are far from being optimal. The basic idea, however, is the same in all these papers: E. Hopf's method is applied to an appropriate exhaustion by compact sets. This introduces boundary terms that have to be controlled in the limit and that do not appear in the case of the $2$-torus treated by E. Hopf. Here the essential difficulty is that the geometric quantities that influence these boundary terms might oscillate dramatically in the non-compact situation. Any naive attempt to control them induces unwanted additional assumptions, as present in the previous results. To our surprise, a delicate analysis of the differential inequality that results from E. Hopf's method finally leads to the optimal results presented here.

In \cite{BK} the same method is applied to complete planes without conjugate points. The rigidity result proved in \cite{BK} assumes a strong ``parallel axiom''. In connection with this and Theorem~\ref{thm:plane} we mention the following interesting question posed by G.\ Knieper, see \cite{CROKE08} and \cite[1.7]{BK}.

\begin{openproblem}
  Suppose a complete Riemannian plane $P$ satisfies the parallel axiom, i.e.\ for every geodesic $c$ on $P$ and every point $p\in P$ not on $c$ there exists a unique geodesic through $p$ that does not intersect $c$. Does this imply that $P$ is isometric to the Euclidean plane?
\end{openproblem}

For complete Riemannian manifolds $(M,g)$ with $\operatorname{dim}M=n\geq 3$ the natural generalization of the area growth is the volume growth $V(M,g)$ defined by
\[
  V(M,g)=\liminf_{r\to\infty}\frac{\operatorname{Vol}(B(p,r))}{\beta(n)\,r^n}
\]
where $\beta(n)$ is the volume of the unit ball in Euclidean $n$-space. For compact manifolds $(M,g)$ one considers the volume growth of the universal Riemannian covering space $(\tilde{M},\tilde{g})$. In the compact situation one has the following beautiful rigidity results.

\begin{theorem-nonumber}[C.~Croke \cite{CROKE92}]
  If $(M,g)$ is a compact, connected Riemannian manifold without conjugate points then $V(\tilde{M},\tilde{g})\geq 1$ with equality if and only if $g$ is flat.
\end{theorem-nonumber}

\begin{theorem-nonumber}[D.~Burago, S.~Ivanov \cite{BI95}]
  If $g$ is a $\mathinteger^n$-periodic Riemannian metric on $\mathreal^n$ then $V(\mathreal^n,g)\geq 1$ with equality if and only if $g$ is flat.
\end{theorem-nonumber}

For complete, noncompact Riemannian manifolds there are results showing that arbitrary compactly supported perturbations of certain natural metrics without conjugate points necessarily introduce conjugate points, see \cite{GG}, \cite{CK98b} and \cite{CK98a}.

Another type of rigidity results for complete, noncompact Riemannian manifolds without conjugate points assumes summability conditions on the Ricci curvature, see \cite{INNAMI86}, \cite{INNAMI89} and \cite{GUIMARAES}.

\vspace{0.4em}
\textbf{Plan of the paper.} In Section~2 we derive the equality that arises from E.~Hopf's method in the case that the compact surface without conjugate points has a boundary. The proof of Theorem~\ref{thm:plane} is given in Section~3. It depends on a differential inequality derived from E.~Hopf's method and a sharp estimate for the asymptotic growth of functions satisfying this differential inequality. This sharp estimate is proved in Section~4. In Sections 5 and 6 we describe the exhaustion by compact subcylinders with horocyclic boundary that we use in the case of cylinders. This is based on results from \cite{BE11}. The proof of Theorem~\ref{thm:l} is given in Section~7. It also depends on the sharp growth estimate from Section~4. Finally, in Section~8 we prove the equivalence of Theorem~\ref{thm:A} and Theorem~\ref{thm:l}. 

\vspace{0.4em}
\textbf{Acknowledgment.} This research was supported by the German Research Foundation (DFG) Collaborative Research Center SFB~TR~71.

\sectionbf{Hopf's method}

\noindent E.~Hopf's method from \cite{HOPF} is also applicable in the more general situation of complete surfaces without conjugate points. Lemma~\ref{lem:gausstheorem} below is used to treat the boundary terms that occur in the non-compact situation. It could be of independent interest, and is stated for arbitrary dimensions.

Let $M$ be a complete Riemannian manifold of dimension $n\geq 2$. Let $\langle,\rangle$ be its metric, let $\sigma\colon T^1M\to M$ be its unit tangent bundle, and let $\geodflow\colon T^1M\times\mathreal\to T^1M$ be its geodesic flow. We say that a function $f\colon T^1M\to\mathreal$ has Lie derivative $\mathcal{L}_\geodflow f\colon T^1M\to\mathreal$ with respect to the flow $\geodflow$, if for every $v\in T^1M$ the function $t\mapsto f\circ\geodflow (v,t)$ is differentiable and
\[
(\mathcal{L}_\geodflow f)(v) = \left. \frac{d}{dt} \right|_{t=0} (f\circ\geodflow)(v,t).
\]
We use the Liouville measure $\mlle$ on $T^1M$, as well as the measure $\mllb$ on $T^1M$ that is locally the product of the codimension-one Hausdorff measure $\mhdf^{n-1}$ on $M$ with the standard Lebesgue measure on the fibers of $T^1M$.

In the proof of Theorem~\ref{thm:l} we will exhaust the cylinder by compact subcylinders with horocyclic boundaries. A priori, these boundaries need not be smooth, and this is why we introduce subsets with strong Lipschitz boundary at this point. By definition, a closed subset $A$ of $M$ has strong Lipschitz boundary $B$, if $A=\cl{\mathring{A}}$ and if $B=A\setminus\mathring{A}$ is a strong Lipschitz submanifold of $M$, see \cite[p.~334]{BANGERT81}. If $B$ is nonempty then $\operatorname{dim} B=n-1$. Let $\mathfrak{A}\defeq \sigma^{-1}(A)$ and $\mathfrak{B}\defeq\sigma^{-1}(B)$. Let $B'$ be the points of differentiability of $B$. By Rademacher's theorem the set $B\setminus B'$ has $\mhdf^{n-1}$-measure zero. Let $\mathfrak{B}'\defeq\sigma^{-1}(B')$, and let $N\colon B'\to\mathfrak{B}'$ denote the inner unit normal.

\begin{lemma}\label{lem:gausstheorem}
  If $A$ is a compact subset of $M$ with strong Lipschitz boundary $B$, if $f\colon T^1M\to\mathreal$ is Borel measurable and has Lie derivative $\mathcal{L}_\geodflow f$, and if $f$ and $\mathcal{L}_\geodflow f$ are locally bounded, then
  \[
    \int_\mathfrak{A} \mathcal{L}_\geodflow f \, d\mlle = -\int_{\mathfrak{B}'} f(v)\, \nv \, d\mllb (v).
  \]
\end{lemma}

We will prove the lemma at the end of this section. Now we present E.~Hopf's method, which is based on the following observation made in \cite{HOPF}. As usual, we let $K$ denote the Gaussian curvature of a Riemannian surface.

\begin{lemma}\label{lem:exu}
Let $S$ be a complete Riemannian surface without conjugate points. Then there exists a Borel measurable and locally bounded function $u\colon T^1S\times\mathreal\to\mathreal$ with the following properties: For every $v\in T^1S$ the function $u_v(t)\defeq u(v,t)$ is a solution of the Riccati equation along the geodesic $\gamma_v$ with $\dot{\gamma}_v(0)=v$, i.e.\
\[
(u_v)'(t)+(u_v)^2(t)+K\circ \gamma_v(t)=0
\]
for all $t\in\mathreal$. The function $u$ is invariant under the geodesic flow $\geodflow$, i.e.\
\[
u(v,s+t)=u(\geodflow^sv,t)
\]
for all $v\in T^1S$ and all $s,t\in\mathreal.$  
\end{lemma}
\begin{proof}
  Except for the local boundedness this follows from \cite{HOPF}. That $u$ is locally bounded also in the non-compact situation is proved in \cite[p.\ 787]{BE11}.
\end{proof}

We express the properties of $u$ in the following form. By the $\geodflow$-invariance of $u$ there exists a function $U\colon T^1S\to\mathreal$ such that
\begin{equation}\label{eqn:exU}
  u=U\circ\geodflow.
\end{equation}
It is given by $U(v)=u(v,0)$ for every $v\in T^1S$. It follows that $U$ is Borel measurable and locally bounded. Since the $u_v$ are solutions of the Riccati equation, it follows that $U$ has Lie derivative $\mathcal{L}_\geodflow U (v) = (u_v)'(0)$ for every $v\in T^1S$, and that $U$ satisfies the equation
\begin{equation}\label{eqn:lie-riccati}
  \mathcal{L}_\geodflow U + U^2 + K\circ\sigma =0
\end{equation}
on $T^1S$. Let $A$ denote the area induced by the Riemannian metric.

\begin{proposition}\label{lem:keylem}
  Let $S$ be a complete Riemannian surface without conjugate points, and let $U$ be as in \eqref{eqn:exU}. If $Q$ is a compact subset of $S$ with strong Lipschitz boundary $\partial Q$, then
  \[
    \int_{\sigma^{-1}(Q)} U^2 \,d\mlle = -2\pi \int_Q K \,dA + \int_{\sigma^{-1}(\partial Q')} U(v)  \, \nv  \,d\mllb (v).
  \]
\end{proposition}

\begin{proof}
We integrate \eqref{eqn:lie-riccati} over $\sigma^{-1}(Q)$ with respect to $\mlle$, and use that the pushforward of $\mlle$ under $\sigma$ is given by $2\pi\area$. Hence we obtain
\[
  \int_{\sigma^{-1}(Q)} U^2 \,d\mlle = -2\pi \int_Q K \,d\area - \int_{\sigma^{-1}(Q)} \mathcal{L}_\geodflow U \,d\mlle.
\]
Applying Lemma~\ref{lem:gausstheorem} to the second summand proves the proposition.
\end{proof}

E. Hopf proves Proposition~\ref{lem:keylem} in case $S$ is a closed surface and $Q=S$. In this case there is no boundary term, see the last equation in \cite{HOPF}. Since in case of the $2$-torus the integral of the Gaussian curvature is zero, the remaining equation implies that $U=0$ almost everywhere. Then \eqref{eqn:lie-riccati} implies that $K=0$, i.e.\ $g$ is flat.

There are previous versions of Proposition~\ref{lem:keylem}, see \cite[1.3]{BK} and \cite[Lem\-ma~4.2]{BE11}. Here, we relax the conditions on the boundary of the subset, and sharpen the result to an equality.

\begin{proof}[Proof of Lemma~\ref{lem:gausstheorem}]
  By definition of the Lie derivative we have the pointwise convergence $\mathcal{L}_\geodflow f=\lim_{\epsilon\to 0} \frac{1}{\epsilon} (f\circ\geodflow^\epsilon - f)$. In particular $\mathcal{L}_\geodflow f$ is Borel measurable. Lebesgue's dominated convergence theorem implies that
\[
  \int_\mathfrak{A} \mathcal{L}_\geodflow f \, d\mlle =\lim_{\epsilon\to 0} \frac{1}{\epsilon}\int_\mathfrak{A}(f\circ\geodflow^\epsilon - f)\, d\mlle.
\]
Since $\geodflow$ preserves $\mlle$ we obtain
\[
  \int_\mathfrak{A} \mathcal{L}_\geodflow f \, d\mlle =\lim_{\epsilon\to 0} \frac{1}{\epsilon}\left( \int_{\geodflow^\epsilon(\mathfrak{A})\setminus\mathfrak{A}} f \, d\mlle - \int_{\mathfrak{A}\setminus\geodflow^\epsilon(\mathfrak{A})} f \, d\mlle \right).
\]
Let $\mathfrak{B}^+\defeq \{ v\in\mathfrak{B}' \colon \nv >0 \}$. We will prove that we have
\begin{equation} \label{eqn:gausstheorem1}
   \lim_{\epsilon\to 0} \frac{1}{\epsilon}\int_{\mathfrak{A}\setminus\geodflow^\epsilon(\mathfrak{A})} f\, d\mlle = \int_{\mathfrak{B}^+} f(v)\, \nv \, d\mllb(v),
\end{equation}
and
\begin{equation}\label{eqn:gausstheorem2}
   \lim_{\epsilon\to 0} \frac{1}{\epsilon}\int_{\geodflow^\epsilon(\mathfrak{A})\setminus\mathfrak{A}} f\, d\mlle = \int_{\mathfrak{B}'\setminus\mathfrak{B}^+} f(v)\, \langle v, - N\circ\sigma(v)\rangle \, d\mllb(v).
\end{equation}
The lemma then follows by combining the preceding three equations.

We first establish formulas \eqref{eqn:af1} - \eqref{eqn:af3} below that are closely related to Santal\'{o}'s formula \cite[eq.~(21)]{SANTALO}. These can be found in \cite[(3.1) - (3.2)]{BANGERT81} in a slightly more special form. We extend the proof that is given in \cite{BANGERT81}. Obviously $\mathfrak{B}\times\mathreal$ is a $(2n-1)$-dimensional strong Lipschitz submanifold of $T^1M\times\mathreal$. Let $\check{\geodflow}\defeq\geodflow\restriction \mathfrak{B}\times\mathreal$. If $J\subseteq \mathfrak{B}\times\mrealpos$ is measurable, then \cite[3.2.5]{FEDERER} implies that
\begin{equation}\label{eqn:af1}
  \mlle(\geodflow(J)) \leq \int_J |\det \check{\geodflow}_*|\, d(\mllb \otimes L^1),
\end{equation}
with equality if $\geodflow\restriction J$ is injective, see \cite[(2.2)]{BANGERT81}. Since $\geodflow$ preserves $\mlle$, we have
\begin{equation}\label{eqn:af2}
  |\det \check{\geodflow}_*| (v,t)=\nv
\end{equation}
for every $(v,t)\in\mathfrak{B}'\times\mrealpos$, see \cite[p.~337]{BANGERT81}. From \cite[3.2.5]{FEDERER} it follows that\break for every measurable $J\subseteq \mathfrak{B}\times\mrealpos$ such that $\geodflow\restriction J$ is injective we have
\begin{equation}\label{eqn:af3}
  \int_{\geodflow(J)} f\, d\mlle = \int_J f\circ\geodflow (v,t) \, \nv\, d(\mllb \otimes L^1).
\end{equation}

We next prove \eqref{eqn:gausstheorem1} and \eqref{eqn:gausstheorem2}. For every $\epsilon>0$ let
\[
  \mathfrak{D}_\epsilon\defeq \{ v\in \mathfrak{B}^+\colon \geodflow^t(v)\in\mathring{\mathfrak{A}} \text{ and } \geodflow^{-t}(v)\in(T^1M\setminus\mathfrak{A})\mathring{} \text{ for } t\in (0,\epsilon)\}.
\]
Then $\bigcup_{\epsilon>0}\mathfrak{D}_\epsilon = \mathfrak{B}^+$, and $\geodflow\restriction \mathfrak{D}_\epsilon\times [0,\epsilon]$ is injective. Note that
\begin{equation} \label{eqn:af4}
  \geodflow(\mathfrak{D}_\epsilon\times (0,\epsilon))
\,\subseteq\, \mathfrak{A}\setminus\geodflow^\epsilon (\mathfrak{A})
\,\subseteq\, \geodflow\left((\mathfrak{B}^+\cup (\mathfrak{B}\setminus\mathfrak{B}'))\times (0,\epsilon)\right)
\end{equation}
for every $\epsilon>0$. The set-theoretic difference of the set on the right and the set on the left is contained in $\geodflow\left(\left((\mathfrak{B}^+\setminus\mathfrak{D}_\epsilon)\cup (\mathfrak{B}\setminus\mathfrak{B}')\right) \times (0,\epsilon)\right)$. By \eqref{eqn:af1} and \eqref{eqn:af2} we have
\begin{align*}
  \lim_{\epsilon\to 0}\frac{1}{\epsilon}
\mlle\left(
  \geodflow\left(
     \left((\mathfrak{B}^+\setminus\mathfrak{D}_\epsilon)\cup (\mathfrak{B}\setminus\mathfrak{B}')\right) \times (0,\epsilon)
  \right) 
\right)\\
\leq \lim_{\epsilon\to 0}\int_{\mathfrak{B}^+\setminus\mathfrak{D}_\epsilon} \nv \, d\mllb(v),
\end{align*}
where the right hand side, and hence the left hand side, is zero. Now it follows from \eqref{eqn:af4} and from \eqref{eqn:af3} that
\begin{align*}
  \lim_{\epsilon\to 0}\frac{1}{\epsilon}\int_{\mathfrak{A}\setminus\geodflow^\epsilon (\mathfrak{A})} f\, d\mlle
&=   \lim_{\epsilon\to 0}\frac{1}{\epsilon}\int_{\geodflow(\mathfrak{D}_\epsilon\times (0,\epsilon))} f\, d\mlle \\
&= \lim_{\epsilon\to 0}\frac{1}{\epsilon}\int_{\mathfrak{D}_\epsilon\times (0,\epsilon)} f\circ\geodflow (v,t)\, \nv\, d(\mllb\otimes L^1).
\end{align*}
Since $\lim_{\epsilon\to 0}\frac{1}{\epsilon} \int_0^\epsilon f\circ\geodflow (v,t) \, dt = f(v)$ for every $v$, Fubini's theorem and Lebesgue's dominated convergence theorem applied to the right hand side prove \eqref{eqn:gausstheorem1}. Equation \eqref{eqn:gausstheorem2} follows from \eqref{eqn:gausstheorem1}, since $\tilde{A}\defeq\cl{M\setminus A}$ is a closed subset of $M$ that has strong Lipschitz boundary $B$, and since the symmetric difference of $\geodflow^\epsilon (\mathfrak{A})\setminus\mathfrak{A}$ and $\tilde{\mathfrak{A}}\setminus\geodflow^\epsilon (\tilde{\mathfrak{A}})$ is contained in $\mathfrak{B}$, and hence has $\mlle$-measure zero.
\end{proof}

\sectionbf{The case of the plane}

\begin{proof}[Proof of Theorem~\ref{thm:plane}.]
Let $U\colon T^1\mathreal^2\to\mathreal$ be as in \eqref{eqn:exU}.

We apply E.\ Hopf's method to the exhaustion that is given by the closed metric balls with center $p$. By Proposition~\ref{lem:keylem} we have
\[
  \int_{\sigma^{-1}(\cl{B(p,r)})} U^2 \,d\mlle \leq -2\pi \int_{B(p,r)} K \,dA + \int_{\sigma^{-1}(\partial B(p,r))} |U| \,d\mllb 
\]
for every $r> 0$. We estimate the boundary term, using the Cauchy-Schwarz inequality and the local product structure of the measure $\mllb$, by
\[
  \int_{\sigma^{-1}(\partial B(p,r))} |U| \,d\mllb\leq \left( \int_{\sigma^{-1}(\partial B(p,r))} U^2\,d\mllb \cdot 2\pi \mhdf^1(\partial B(p,r))\right)^\frac{1}{2}.
\]

The coarea formula \cite[3.2.11]{FEDERER}, applied to the distance function from $p$, implies that the area of the metric ball $B(p,r)$ satisfies
\begin{equation}\label{eqn:A=inth}
  A_p(r)=\int_0^r \mhdf^1 (\partial B(p,t))\, dt
\end{equation}
for every $r> 0$. The function $A_p\colon\mrealpos\to\mrealpos$ is smooth and non-decreasing. Similarly, if we define $F_p\colon\mrealpos\to\mrealnneg$ at $r$ as the left hand side of the first of the above two inequalities, then
\[
  F_p(r)\defeq\int_{\sigma^{-1}(\cl{B(p,r)})} U^2 \,d\mlle=\int_0^r  \left(\int_{\sigma^{-1}(\partial B(p,t))}U^2\,d\mllb\right) dt
\]
for every $r> 0$, by the coarea formula in $T^1\mathreal^2$. The function $F_p$ is non-decreasing, and by \cite[2.9.20]{FEDERER} it is locally absolutely continuous.

We combine the preceding two inequalities, and obtain
\[
F_p(r)   \leq   -2\pi \int_{B(p,r)} K\, dA   + \left(F'_p(r) \cdot 2\pi A_p'(r) \right)^\frac{1}{2}
\]
for almost every $r> 0$. From \eqref{eqn:A=inth} and the first variation formula it follows that the second derivative of $A_p$ at $r$ equals the rotation of $\partial B(p,r)$, i.e.\ the integral of the geodesic curvature of $\partial B(p,r)$ with respect to the inward pointing normal. Hence the Gauss-Bonnet theorem implies
\[
  A_p''(r) = 2\pi - \int_{B(p,r)} K\, dA 
\]
for every $r> 0$. From the preceding two expressions we obtain a differential inequality for the functions $A_p$ and $F_p$, valid almost everywhere on $\mrealpos$:
\[
  F_p \;\leq\; 2 \pi A_p'' + \sqrt{2\pi}\big(F_p'\, A_p'\big)^\frac{1}{2} -4\pi^2.
\]

We apply Lemma~\ref{lem:analysis} below to this inequality. It is stated in a slightly more general form, since we will also apply it in the proof of Theorem~\ref{thm:l}. We let $A=A_p$, $F=F_p$, $R=A_p''$, $a=2\pi$, $b=\sqrt{2\pi}$, $c=-4\pi^2$, and obtain
\[
  0 \leq \sup F_p \leq 4\pi \left( \liminf_{r\to\infty}\frac{A_p(r)}{r^2} -\pi\right).
\]
This implies the inequality $\liminf_{r\to\infty}\frac{A_p(r)}{\pi r^2}\geq 1$ claimed in Theorem~\ref{thm:plane}. If equality holds we obtain $F_p= 0$. The remainder of the proof is the same as in E.\ Hopf's original argument: By definition of $F_p$, we have $U=0$ almost everywhere, and then \eqref{eqn:lie-riccati} implies that $K=0$, i.e.\ $g$ is flat.
\end{proof}

\sectionbf{A sharp estimate for an ODE inequality}

\noindent{}Here we provide the analysis of the differential inequality that results from E.\ Hopf's method applied to the plane and the cylinder.

\begin{lemma}\label{lem:analysis}
  Let $A,F\colon\mrealpos\to\mrealnneg$ be non-decreasing and locally absolutely continuous, and suppose that $R\in L^1_{\operatorname{loc}}(\mrealpos,\mathreal)$ satisfies\break $\int_0^r \left(\int_0^\rho R(t) \,d t\right) d\rho \leq A(r)$ for every $r>0$. If $a,b>0$ and $c\in\mathreal$ are constants such that the differential inequality 
\[
  F\leq a \, R + b\,(F'A')^\frac{1}{2} + c
\]
holds almost everywhere on $\mrealpos$, then
\[
  \sup F\leq 2 a\, \liminf_{r\to\infty}\frac{A(r)}{r^2} +c.\]
\end{lemma}

\begin{proof}
In the course of the proof we wish to integrate the differential inequality twice. Let $[q,r]\subseteq \mrealpos$ be a nonempty compact interval. Define
\[
  I(q,r)\defeq \int_q^r \left( \int_q^\rho F(t)-F(q)\, dt\right) d\rho.  
\]
First, we prove that
\begin{equation}\label{eqn:ineqR}
  \int_q^r\left( \int_q^\rho (F'A')^\frac{1}{2}(t)\, dt\right) d\rho \leq \big( 2 I(q,r) A(r) \big)^\frac{1}{2}.
\end{equation}
Integration by parts \cite[2.9.24]{FEDERER} shows that for every Lebesgue-integrable function $f:[q,r]\to\mathreal$ we have
\begin{equation}\label{eqn:intbyparts}
  \int_q^r \left( \int_q^\rho f(t)\, dt\right) d\rho = \int_q^r (r-\rho)\, f(\rho)\, d\rho.
\end{equation}
We use this equation for $f\defeq (F'A')^\frac{1}{2}$ and apply the Cauchy-Schwarz inequality, to obtain
\begin{align*}
  \int_q^r\left( \int_q^\rho\, (F'A')^\frac{1}{2}(t) \,dt\right) d\rho \leq \left( \int_q^r (r-\rho)^2 F'(\rho)\,d\rho\right)^\frac{1}{2} A(r)^\frac{1}{2}.
\end{align*}
The content of the first bracket on the right hand side equals $2I(q,r)$. This follows from integration by parts, from equation \eqref{eqn:intbyparts} for $f\defeq F$, and from
\[
  2 I(q,r) = 2 \int_q^r \left( \int_q^\rho F(t)\, dt\right) d\rho \;-(r-q)^2 F(q).
\]
This proves \eqref{eqn:ineqR}.

  Integrating the differential inequality twice, we obtain
\begin{align*}
\int_q^r\left( \int_q^\rho F(t)\, dt\right) d\rho \;\leq\;& a \,\int_q^r\left( \int_q^\rho R(t)\, dt\right) d\rho  \\
+ &b\int_q^r\left( \int_q^\rho (F'A')^{\frac{1}{2}}(t)\, dt\right) d\rho \;+\; c\,\frac{(r-q)^2}{2}.
\end{align*}
On the left hand side we use the preceding equation, on the right hand side we use the assumption on $R$ and \eqref{eqn:ineqR}. It follows that
\[
  I(q,r) + \frac{F(q)}{2}r^2 \;\leq\;a\, A(r) + b \, \big( 2 I(q,r) A(r) \big)^\frac{1}{2} + \frac{c}{2}\,r^2 + \mathcal{O}(q,r)
\]
for every $0< q<r<\infty$ and for a function $\mathcal{O}(q,r)$ that satisfies\break $\lim_{r\to\infty}\mathcal{O}(q,r)/r^2=0$ for every fixed $q> 0$. We divide by $r^2$ and obtain
\begin{equation}\label{eqn:qr-ineq}
 \frac{I(q,r)}{r^2} + \frac{F(q)}{2}\;\leq\;a\,\frac{A(r)}{r^2}+  b\, \left( 2\frac{I(q,r)}{r^2}\right)^\frac{1}{2}\left( \frac{A(r)}{r^2}\right)^\frac{1}{2} + \frac{c}{2} + \frac{\mathcal{O}(q,r)}{r^2}
\end{equation}
for every $0< q<r<\infty$.

Since $F$ is non-decreasing, we have
\[
  \lim_{r\to\infty}\frac{I(q,r)}{r^2}=\frac{\sup F-F(q)}{2}.
\]
To prove the lemma, we can assume that $\liminf_{r\to\infty} A(r)/r^2$ is finite. Then \eqref{eqn:qr-ineq} implies that the preceding expression, i.e.\ $\sup F$, is finite. Now take the limit inferior $r\to\infty$ of \eqref{eqn:qr-ineq} to obtain
\[
  \frac{\sup F}{2} \;\leq\; a \liminf_{r\to\infty}\frac{A(r)}{r^2} + b \, \Big(\sup F - F(q)\Big)^\frac{1}{2} \left( \liminf_{r\to\infty}\frac{A(r)}{r^2}\right)^\frac{1}{2} + \frac{c}{2}.
\]
Since $F$ is non-decreasing and $\sup F$ is finite, we have $\lim_{q\to\infty} (\sup F-F(q))=0$. Hence, taking the limit $q\to\infty$ in the preceding inequality proves the lemma.
\end{proof}

\sectionbf{The cylinder: Construction of the exhaustion}

\noindent In the proof of Theorem~\ref{thm:l}, E.~Hopf's method is applied to compact sets exhausting $S^1\times\mathreal$. In \cite{BK}, K.~Burns and G.~Knieper used compact sets bounded by non-contractible geodesic loops. In \cite{BE11} we used sets bounded by horocycles, i.e.\ by level sets of Busemann functions, and this led to better estimates. Here we briefly describe the construction of the exhaustion that we will use in the proof of Theorem~\ref{thm:l}; it is similar to the exhaustion in \cite{BE11}.

\begin{definition}[{\cite[(22.3)]{BUSEMANN}}]
  Let $\gamma$ be a ray in a complete Riemannian manifold $M$. Its \emph{Busemann function} $b_\gamma\colon M\to\mathreal$ is defined by
  \[
   b_\gamma(p) \defeq \lim\limits_{t\to\infty}\big(\dist(p,\gamma(t))-t\big).
  \]
\end{definition}

Busemann functions are examples of distance functions as defined in \cite{GROVE}, see also \cite[Section 2]{BE11}. Distance functions are Lipschitz, and not $C^1$ in general, but there is a notion of regularity for them. It has similar consequences as the usual one. In particular, all level sets of a proper and regular distance function are homeomorphic, see \cite[Proposition~1.8]{GROVE}.

\begin{proposition}[{\cite[Proposition 3.2]{BE11}}]\label{prp:regular}
  Let $C=S^1\times\mathreal$ be a complete Rie\-mannian cylinder without conjugate points. If $\gamma\colon\mrealnneg\to C$ is a ray such that
  \[
    \liminf_{t\to\infty}\frac{1}{t} \, l(\gamma(t))<1
  \]
then $b_\gamma$ is a proper and regular distance function, and each of its level sets is homeo\-morphic to $S^1$ and generates the fundamental group of $C$.
\end{proposition}

Now assume that $C=S^1\times\mathreal$ is a complete Riemannian cylinder without conjugate points such that both ends of $C$ open less than linearly. Choose a minimal geodesic $\gamma\colon\mathreal\to C$ that joins the two ends, see \cite[p.\ 630]{BK}. Consider the two rays $\gamma_1(t)\defeq\gamma(-t)$ $(t\geq 0)$ and $\gamma_2(t)\defeq\gamma(t)$ $(t\geq 0)$ of $\gamma$. Then $\gamma_1$ and $\gamma_2$ converge to the different ends of $C$. Our exhaustion will be constructed from the Busemann functions $b_{\gamma_1}$ and $b_{\gamma_2}$. Since both ends of $C$ open less than linearly, the assumption of Proposition~\ref{prp:regular} is satisfied for the rays $\gamma_1$ and $\gamma_2$. Since $\gamma$ is minimal we have
\[
  b_{\gamma_1}^{-1}((-\infty,0))\cap b_{\gamma_2}^{-1}((-\infty,0))=\emptyset.
\]

\begin{definition}
Let $H^0$ be the complement of $b_{\gamma_1}^{-1}((-\infty,0))\cup b_{\gamma_2}^{-1}((-\infty,0))$ in $C$. For $i\in\{1,2\}$ and every $t>0$, we define the sets
\[
H^i_t\defeq b_{\gamma_i}^{-1}([-t,0)),\qquad  h^i_t\defeq b_{\gamma_i}^{-1}(-t),
\]
and the functions
\[
H_i(t)\defeq A(H^i_t),\qquad  h_i(t)\defeq\mhdf^1(h^i_t).
\]
\end{definition}

Here the notation is chosen so that $\gamma_i(t)\in h^i_t$ for $i\in\{1,2\}$ and $t > 0$.

\begin{proposition}\label{prp:exhaustion}
For every pair of positive real numbers $(r_1,r_2)$, the set $H^1_{r_1}\cup H^0\cup H^2_{r_2}$ is a compact subcylinder of $C$ with strong Lipschitz boundary $h^1_{r_1}\cup h^2_{r_2}$. For $i\in\{1,2\}$ and every $r> 0$ we have
\begin{equation}\label{eqn:equidist}
    h^i_t=\{p\in H^i_r \colon d(p,h^i_r)=r-t\} \qquad (0<t<r),
\end{equation}
  i.e.\ the horocycles $h^i_t$ are ``inner equidistants'' of $h^i_r\subseteq \partial H^i_r$. For $i\in\{1,2\}$, $h_i\colon\mrealpos\to\mrealnneg$ is continuous, and $H_i\in C^1(\mrealpos,\mrealpos)$ satisfies $H_i'=h_i$.
\end{proposition}

\begin{proof}
  Let $i\in\{1,2\}$. By Proposition \ref{prp:regular}, $b_{\gamma_i}$ is a proper and regular distance function and its level sets $h^i_t$, $t\in\mathreal$, are circles that generate the fundamental group of $C$. For every $r\in\mathreal$ and every $p\in C$ such that $b_{\gamma_i}(p)\geq r$ we have
\[
  b_{\gamma_i}(p)= d(p,b_{\gamma_i}^{-1}(r))+r,
\]
 which is a general fact for Busemann functions, and follows from \cite[(22.18)]{BUSEMANN}. Hence, for every $r>0$ and every $p\in H^i_r$, we have
\[
  b_{\gamma_i}(p)=d (p,h^i_r)-r.
\]
This implies \eqref{eqn:equidist}.  From \cite[Proposition~2.1]{BE11} it follows that the $h^i_t$ are strong Lipschitz submanifolds of $C$ and that $h_i$ is continuous. Since $b_{\gamma_i}$ is a distance function, it satisfies $|\operatorname{grad} b_{\gamma_i}|=1$ almost everywhere. Hence $H_i'=h_i$ by the coarea formula \cite[3.2.11]{FEDERER} applied to the function $b_{\gamma_i}$.
\end{proof}

Note that so far we used only that the assumption of Proposition \ref{prp:regular} is satisfied for the rays $\gamma_1$ and $\gamma_2$, and not the stronger assumption that both ends of the cylinder open less than linearly. The latter one now has the following effect on the growth of the area functions $H_i$, which will be used in the proof of Theorem~\ref{thm:l}.

\begin{lemma}\label{lem:subqH}
  If $\gamma_i$ converges to an end $\mathcal{E}$ that opens less than linearly, then
\[
  \liminf_{r\to\infty}\frac{H_i(r)}{r^2}=0.
\]
\end{lemma}

\begin{proof}
Since $\mathcal{E}$ opens less than linearly, there exists a sequence $(p_j)$ in $C$ converging to $\mathcal{E}$ such that
\[
  \lim_{j\to\infty}\frac{l(p_j)}{d(p_j,p_0)}=0.
\]
Shortest non-contractible loops on $C$ are simple and generate the fundamental group of $C$, see \cite[Remark 3.1]{BE11}. Since the ray $\gamma_i$ and the sequence $(p_j)$ converge to the same end, a shortest non-contractible loop with basepoint $p_j$ intersects the ray $\gamma_i$, for all but possibly finitely many $j$. Hence  we can assume that $p_j=\gamma_i(s_j)$ for a sequence $s_j\to\infty$. In addition, we can assume $p_0=\gamma_i(0)$, so that $s_j=\dist (p_j,p_0)$ for every $j$.

For every $j\ge 1$ we choose a shortest non-contractible loop $\Gamma_j$ based at $p_j$, and let $r_j\defeq s_j-l(p_j)$. We claim that, for every $j\geq 1$, we have
\[
  \Gamma_j\subseteq b_{\gamma_i}^{-1}((-\infty,-r_j)).
\]
To prove this, note, that at the basepoint of $\Gamma_j$ we have $b_{\gamma_i}(p_j)=-s_j$, that $b_{\gamma_i}$ is $1$-Lipschitz, and that $\operatorname{diam}(\Gamma_j)\leq l(p_j)/2$. Choose a non-contractible, simple, piecewise $C^1$-regular loop $\Gamma_0$ based at $p_0$ that is contained in $b_{\gamma_i}^{-1}([0,\infty))$. It follows that, for every $j$ such that $r_j> 0$, the set $H^i_{r_j}=b_{\gamma_i}^{-1}([-r_j,0))$ is contained in the compact subcylinder that is bounded by $\Gamma_0$ and $\Gamma_j$. We use \cite[Proposition~7.2]{BE11} to estimate the area of this subcylinder and hence the area of the included set $H^i_{r_j}$ by
\[
  H_i(r_j)\leq\frac{8}{\pi} (s_j+ L_j) L_j
\]
for $L_j\defeq \frac{1}{2}\big( \operatorname {length}(\Gamma_0)+l(p_j)\big)$. The condition on the sequence $(p_j)$ implies $\lim_{j\to\infty} l(p_j)/s_j =0$. It follows that $\lim_{j\to\infty}r_j/s_j=1$ and then, from the preceding inequality, that $\lim_{j\to\infty} H_i(r_j) / (r_j)^2 =0$. 
\end{proof}

\sectionbf{The cylinder: The lengths of inner equidistants}

\noindent{}We continue to assume that the rays $\gamma_1$, $\gamma_2$ and the associated exhaustion of $C$ are as in the previous section.

Let $i\in\{1,2\}$. For technical reasons we choose an isometric embedding of the half cylinder $b_{\gamma_i}^{-1}((-\infty,0))$ into some complete Riemannian $\mathreal^2$. In this way we attach a closed Riemannian disk $D$ to the half cylinder $b_{\gamma_i}^{-1}((-\infty,0))$. It follows that the union of $H^i_t=b_{\gamma_i}^{-1}([-t,0))$ and $D$ is a closed disk with boundary $h^i_t$, for every $t > 0$. Let $\omega(B)\defeq \int_B K\, d\area$ denote the integral of the Gaussian curvature of a Borel subset $B\subseteq\mathreal^2$. Define
\[
  \omega_i(t)\defeq \omega (D\cup H^i_t)  
\]
for every $t> 0$. The Gauss-Bonnet theorem \cite[2.1.5]{BZ} implies that $\omega_i(t)$ is independent of the geometry of the disk $D$. 

If $b_{\gamma_i}$ is smooth, then  $2\pi - \omega_i(t)$ equals the rotation of $h^i_t=\partial (D\cup H^i_t)$, i.e.\ the integral of its geodesic curvature with respect to the inward pointing normal. By \eqref{eqn:equidist}, the $h^i_t$ are inner equidistants of $h^i_r=\partial(D\cup H^i_r)$ for every $0<t<r$. Hence in the smooth case we have
\[
  h_i'(t)=2\pi-\omega_i(t).
\]
This equation is no longer true if $b_{\gamma_i}$ is not smooth, see \cite[p.~791]{BE11}. In the general case one has the following inequality that goes back to work by G. Bol \cite{BOL} and F.~Fiala \cite{FIALA}. The proof given below consists in applying \cite[3.2.3]{BZ} to our situation.

\begin{lemma}\label{lem:71}
For $i\in\{1,2\}$ and every $r > 0$ we have
  \[
    h_i(r)\geq \int_0^r 2\pi - \omega_i(t)\, dt.
  \]
\end{lemma}

Note that the direction of this inequality is compatible with inequality \eqref{eqn:diffineq-cyl} below. This is due to our choice of exhaustion.

\begin{proof}[Proof of Lemma~\ref{lem:71}]
Let $G\defeq D\cup H^i_r$. Then $G$ is a subset of a complete Riemannian plane, $G$ is homeomorphic to a closed disk and has strong Lipschitz boundary. For $0< t < r$ let
\[
  P_t\defeq \{p\in G\colon d(p,\partial G)<t\},\qquad l_t\defeq \{p\in G\colon d(p,\partial G)=t\}.
\]
By \eqref{eqn:equidist} we have, for every $0<t<r$,
\[
  H^i_t=\bigcup_{0< s\leq t} h^i_s = \{ p\in H^i_r\colon d(p,h^i_r)\geq r-t\}.
\]
Thus $G$ is the disjoint union of $D\cup H^i_t$ and $P_{r-t}$, for every $0< t < r$. Hence
\[
  \omega_i(t)+\omega(P_{r-t})=2\pi - \tau,
\]
for $\tau\defeq 2\pi - \omega(G)$. Note that $\tau$ is the rotation of $\partial G$ as defined in \cite[2.1.5]{BZ}. Together with $l_0=h^i_r$ this implies that the inequality in Lemma~\ref{lem:71} is equivalent to the inequality
\[
  \mhdf^1(l_0) \geq \int_0^r \omega( P_t) + \tau \, dt.
\]
According to the approximation theorem \cite[3.1.1]{BZ} it suffices to prove the preceding inequality in case $G$ is a polyhedron as in \cite[3.1]{BZ}. In this case an application of lemma \cite[3.2.3]{BZ} shows that
\[
  \mhdf^1 (l_r) - \mhdf^1(l_0) \leq \int_0^r -\omega(\cl{P_t}) -\tau \, dt.
\]
\end{proof}

The reader may consult the appendix to \cite{BE11} for details on the approximation and for the application of lemma \cite[3.2.3]{BZ}.

\sectionbf{The cylinder: Proof of Theorem~\ref{thm:l}}

\noindent{}The proof of Theorem~\ref{thm:l} is similar to the proof of Theorem~\ref{thm:plane}. Some additional arguments, provided in Sections~5 and 6, are necessary.

\begin{proof}[Proof of Theorem~\ref{thm:l}]
Let $U\colon T^1C\to\mathreal$ be as in \eqref{eqn:exU}.

We apply E.~Hopf's method to an exhaustion by compact subcylinders as in Proposition~\ref{prp:exhaustion}. By Proposition~\ref{lem:keylem} we have
\[
 \int_{\sigma^{-1}(H^1_{r_1}\cup H^0\cup H^2_{r_2})} U^2 \,d\mlle \leq -2\pi \int_{H^1_{r_1}\cup H^0\cup H^2_{r_2}} K \,dA + \int_{\sigma^{-1}(h^1_{r_1}\cup h^2_{r_2})} |U|\,d\mllb
\]
for every $(r_1,r_2)\in \mrealpos\times \mrealpos$. We estimate the boundary terms, using the Cauchy-Schwarz inequality and the local product structure of $\mllb$, by
\[
   \int_{\sigma^{-1}(h^i_r)} |U|\,d\mllb \leq \left(\int_{\sigma^{-1}(h^i_r)}U^2\, d\mllb  \cdot 2\pi \mhdf^1(h^i_r)\right)^\frac{1}{2}
\]
for every $i\in\{1,2\}$ and every $r>0$. 

Recall that $H_i(r)=A(H^i_r)$ and $h_i(r)=\mhdf^1(h^i_r)$, and that $H_i'=h_i$ by Proposition~\ref{prp:exhaustion}. The coarea formula \cite[3.2.12]{FEDERER} in $T^1C$ implies that
\[
F_i(r)\defeq\int_{\sigma^{-1}(H^i_r)} U^2\, d\mlle=\int_0^r  \left(\int_{\sigma^{-1}(h^i_t)} U^2\, d\mllb \right) dt.
\]
The functions $F_i:\mrealpos\to\mrealnneg$ are non-decreasing, and locally absolutely continuous by \cite[2.9.20]{FEDERER}. Let $F_0\defeq\int_{\sigma^{-1}(H^0)} U^2\, d\mlle$.

We combine the preceding two inequalities, and obtain
\[
  F_0 + \sum_{i\in\{1,2\}} F_i({r_i})\leq -2\pi\, \int_{H^1_{r_1}\cup H^0\cup H^2_{r_2}} K\, dA +\sum_{i\in\{1,2\}} \big(F_i'(r_i)\cdot 2\pi H_i'(r_i)\big)^\frac{1}{2}
\]
for almost every $({r_1},{r_2})\in \mrealpos\times \mrealpos$. We claim that, for $\omega_i$ as in Section~6,
\[
\int_{H^1_{r_1}\cup H^0\cup H^2_{r_2}} K\, dA=\omega_1({r_1})  +  \omega_2({r_2})-4\pi.
\]
for every $(r_1,r_2)\in\mrealpos\times\mrealpos$. To prove this, choose a smooth, simple closed, non-contractible curve $c$ that is contained in the interior of the compact subcylinder $H^1_{r_1}\cup H^0\cup H^2_{r_2}$. Choose a Riemannian metric on $S^2$ such that a neighborhood of its equator is isometric to a neighborhood of $c$. Now cut the subcylinder along $c$, and attach the appropriate hemisphere of the $S^2$ to each of the halves, according to the chosen isometry. This process increases the curvature integral by $2\pi\chi(S^2)=4\pi$. Together with the definitions of the $\omega_i$ this proves the equation.

We combine the two expressions, and obtain
\begin{equation}\label{eqn:diffineq-cyl}
  F_0 + \sum_{i\in\{1,2\}} F_i({r_i})\leq 2\pi \sum_{i\in\{1,2\}} \big(2\pi - \omega_i(r_i)\big) +\sqrt{2\pi} \sum_{i\in\{1,2\}} \big(F_i'H_i'\big)^\frac{1}{2}(r_i) 
\end{equation}
for almost every $(r_1,r_2)\in\mrealpos\times\mrealpos$. From $H_i'=h_i$ and Lemma~\ref{lem:71} it follows that
\begin{equation}\label{eqn:condr-cyl}
  \int_0^r \left(\int_0^\rho 2\pi - \omega_i(t)\, dt \right) d\rho \leq H_i(r) 
\end{equation}
for $i\in\{1,2\}$ and every $r> 0$.

Considering the preceding two inequalities, we are precisely in the situation of Lemma~\ref{lem:analysis}: Fix an $r_1>0$ such that \eqref{eqn:diffineq-cyl} is valid for almost every $r_2>0$. We apply the lemma for $A=H_2$, $F=F_2$, $R=2\pi-\omega_2$, $a=2\pi$, $b=\sqrt{2\pi}$, and the constant $c$ defined by
\[
 c(r_1)\defeq 2\pi \big(2\pi-\omega_1(r_1)\big) +\sqrt{2\pi}\big(F_1'H_1'\big)^\frac{1}{2}(r_1)-F_0-F_1(r_1).
\]
According to inequality \eqref{eqn:condr-cyl} the assumption of Lemma~\ref{lem:analysis} is satisfied for $A=H_2$ and $R=2\pi-\omega_2$. Thus we obtain
\[
\sup F_2\leq 4\pi \liminf_{r\to\infty}\frac{H_2(r)}{r^2}+ c(r_1).
\]
This can be done for almost every $r_1>0$. Hence we obtain
\[
  F_1\leq 2\pi (2\pi-\omega_1) +\sqrt{2\pi}\big(F_1'H_1'\big)^\frac{1}{2} + \left(4\pi\liminf_{r\to\infty}\frac{H_2(r)}{r^2}-\sup F_2-F_0\right)
\]
almost everywhere on $\mrealpos$. As above, we apply Lemma~\ref{lem:analysis} for $A=H_1$, $F=F_1$, $R=2\pi-\omega_1$, $a=2\pi$, $b=\sqrt{2\pi}$, and $c$ equal to the value of the rightmost bracket of the preceding inequality. Thus we obtain
\[
\sup F_1\leq 4\pi \liminf_{r\to\infty}\frac{H_1(r)}{r^2}+4\pi\liminf_{r\to\infty}\frac{H_2(r)}{r^2}-\sup F_2-F_0.
\]
By definition of the $F_i$ this is equivalent to
\[
\int_{\sigma^{-1}(C)} U^2\, d\mlle \leq 4\pi \left(\liminf_{r\to\infty}\frac{H_1(r)}{r^2} + \liminf_{r\to\infty}\frac{H_2(r)}{r^2}\right).
\]
Since both ends of $C$ open less than linearly, the right hand side is zero by Lemma \ref{lem:subqH}. It follows that $U=0$ almost everywhere, and then \eqref{eqn:lie-riccati} implies that $K=0$, i.e.\ $g$ is flat.
\end{proof}

\sectionbf{The cylinder: Area growth versus length growth}

\noindent{}Here we prove the equivalence of Theorem~\ref{thm:A} and Theorem~\ref{thm:l}.

\begin{proposition}\label{prp:opening}
  Let $C=S^1\times\mathreal$ be a complete Riemannian cylinder without conjugate points. An end $\mathcal{E}$ of $C$ has subquadratic area growth if and only if it opens less than linearly.
\end{proposition}

We note that the only-if part is true without the assumption of no conjugate points, compare the proof given below. We start with some preparations for the proof of the proposition. Recall that $l(p)$ denotes the length of a shortest non-contractible loop based at $p$, and that $\mhdf^1$ denotes one-dimensional Hausdorff measure on $C$.

\begin{lemma}\label{lem:l-B}
  If $\gamma\colon\mrealnneg\to C$ is a ray that converges to an end $\mathcal{E}$ of $C$, and if $U\subseteq C$ is a neighborhood of $\mathcal{E}$, then
  \[
    l(\gamma(r)) \leq\mhdf^1(U \cap \partial B(\gamma(0),r))
  \]
for every sufficiently large $r>0$.
\end{lemma}

The idea of the proof of the lemma is that, for sufficiently large $r$, the component of $\partial B(\gamma(0),r)$ that contains $\gamma(r)$ is a non-contractible loop. This is true, but it is technically simpler to base the proof on some simple point set topology. The proof is similar to the proof of \cite[Lemma~6.2]{BE11}.

\begin{proof}[Proof of Lemma~\ref{lem:l-B}]
  Choose $r_0>\frac{1}{2} l(\gamma(0))$ such that the component of $C\setminus \cl{B(\gamma(0),r_0)}$ that is determined by $\mathcal{E}$ is contained in $U$. Let $r>r_0$. The minimality of $\gamma$ implies that for every $s>r$ we have $\gamma(s)\in C\setminus \cl{B(\gamma(0),r)}$. Since $\gamma$ converges to $\mathcal{E}$, we even have
\[
    \gamma(s)\in U\setminus \cl{B(\gamma(0),r)}
\]
for every $s>r$. We denote by $E$ the component of $U\setminus \cl{B(\gamma(0),r)}$ containing $\gamma((r,\infty))$. Note that $\cl{E}\subseteq U$, and that $\partial E$ is a closed subset of $U\cap \partial B(\gamma(0),r)$. So our claim is proved once we have proved that
\[
  l(\gamma(r))\leq\mhdf^1 (\partial E).
\]

We prove this by showing that for every $0<t<\frac{1}{2} l(\gamma(r))$ the metric circle $\partial B (\gamma(r),t)$ intersects $\partial E$ in at least two points. Then we can apply \cite[2.10.11]{FEDERER} to the distance function from $\gamma(r)$ and to the set $\partial E$, and obtain $2\cdot\frac{1}{2} l(\gamma(r))\leq\mhdf^1(\partial E)$. To prove that $\partial B(\gamma(r),t)\cap \partial E$ contains at least two points, we note that the minimality of $\gamma$ and the definition of $E$ imply that
\[
 \gamma (r+t)\in\partial B(\gamma(r),t)\cap E \quad \text{and} \quad \gamma (r-t)\in\partial B(\gamma(r),t)\cap (C\setminus \cl{E}).
\]
Since $t<\frac{1}{2} l(\gamma(r))$  and $\frac{1}{2} l(\gamma(r))$ is the injecticity radius of $C$ at $\gamma(r)$, we know that $\partial B(\gamma(r),t)$ is homeomorphic to a circle. Since $E$ is open and $\partial B(\gamma(r),t)\cap E\neq \emptyset$ and $\partial B(\gamma(r),t)\cap (C\setminus \cl{E})\neq\emptyset$, we see that $\partial B(\gamma(r),t)\cap \partial E$ contains at least two points.
\end{proof}

\begin{proof}[Proof of Proposition \ref{prp:opening}]
   We first prove that if an end $\mathcal{E}$ of $C$ has subquadratic area growth then it opens less than linearly.

Let $p\in C$ and let $U\subseteq C$ be Borel measurable. We apply the coarea formula \cite[3.2.11]{FEDERER} to the distance function from $p$, and obtain
\begin{align*}
  A(U\cap B(p,2r)) &\geq \int_r^{2r} \mhdf^1 (U\cap \partial B(p,t))\, dt
\intertext{for every $r> 0$. This implies that}
  \frac{A(U\cap B(p,2r))}{r^2} &\geq  \inf_{t\geq r} \frac{\mhdf^1 (U\cap \partial B(p,t))}{t}
\end{align*}
for every $r> 0$. By our assumption on $\mathcal{E}$, we can choose a neighborhood $U$ of $\mathcal{E}$ such that the limit inferior $r\to\infty$ of the left hand side is zero, and hence
\[
  \liminf_{r\to\infty} \frac{\mhdf^1 (U\cap \partial B(p,r))}{r}=0.
\]
Choose a ray $\gamma\colon\mrealnneg\to C$ that satisfies $\gamma(0)=p$ and that converges to $\mathcal{E}$. Now Lemma~\ref{lem:l-B} and the preceding equation imply that
\[
  \liminf_{r\to\infty} \frac{l(\gamma(r))}{r}=0.
\]
Using the minimality of $\gamma$ we see that $\mathcal{E}$ opens less than linearly.

We next prove the reverse implication, i.e.\ if an end $\mathcal{E}$ of $C$ opens less than linearly then $\mathcal{E}$ has subquadratic area growth. The proof is the same as the proof of Lemma~\ref{lem:subqH}, up to the following modification.

Choose a ray $\gamma_i$ that converges to $\mathcal{E}$. Let the sequences $(p_j)$, $(s_j)$ and $(r_j)$, and the loops $\Gamma_j$, $j\geq 1$, be as in the proof of Lemma~\ref{lem:subqH}. Let $\Gamma_0$ be a shortest non-contractible loop based at $p_0$, and let $U$ be the component of $C\setminus \Gamma_0$ that is a neighborhood of $\mathcal{E}$. By the triangle inequality we have
\[
\Gamma_j\subseteq U\setminus B(p_0,r_j)
\]
for all $j$ such that $r_j>0$. Hence for such $j$ the set $U\cap B(p_0,r_j)$ is contained in the compact subcylinder that is bounded by $\Gamma_0$ and $\Gamma_j$. Now we can apply \cite[Proposition~7.2]{BE11}, as in the proof of Lemma~\ref{lem:subqH}, to obtain $\lim_{j\to\infty}A(U\cap B(p_0,r_j))/(r_j)^2=0$ for a sequence $r_j\to\infty$. This shows that $\mathcal{E}$ has subquadratic area growth.
\end{proof}

\providecommand{\bysame}{\leavevmode\hbox to3em{\hrulefill}\thinspace}
\renewcommand{\MR}[1]{%
  \href{http://www.ams.org/mathscinet-getitem?mr=#1}{MR~#1}
}

\vspace{0.5em}
\small
{\scshape\noindent{}Mathematisches Institut\\
Universit\"at Freiburg\\
Eckerstr.~1\\
79104 Freiburg\\
Germany\\[0.25em]
\email{\textit{E-mail address}: {\normalfont victor.bangert@math.uni-freiburg.de}}}
\\[0.75em]
{\scshape\noindent{}Mathematisches Institut\\
Universit\"at Freiburg\\
Eckerstr.~1\\
79104 Freiburg\\
Germany\\[0.25em]
\email{\textit{E-mail address}: {\normalfont patrick.emmerich@math.uni-freiburg.de}}}

\end{document}